\documentclass[12pt]{article}
\usepackage{amsmath, amssymb}
\usepackage{amsfonts, ifthen, amsthm, vmargin}
\usepackage[usenames]{color}
\usepackage{ulem}
\newcommand{\dotcup}{\mathbin{\mathaccent\cdot\cup}}
\DeclareFontFamily{OT1}{rsfs}{} \DeclareFontShape{OT1}{rsfs}{m}{n}{
<-7> rsfs5 <7-10> rsfs7 <10-> rsfs10}{}
\DeclareMathAlphabet\mathcurl{OT1}{rsfs}{m}{n}

\newtheorem{theorem}{Theorem}[section]
\newtheorem{lemma}[theorem]{Lemma}
\newtheorem{corollary}[theorem]{Corollary}
\newtheorem{definition}[theorem]{Definition}
\newtheorem{proposition}[theorem]{Proposition}

\newtheorem{remark}[theorem]{Remark}

\newtheorem{assumption}[theorem]{Assumption}

\theoremstyle{plain}
\newtheorem*{Sklar}{Sklar's Theorem}

\theoremstyle{plain}

\theoremstyle{plain}
\newtheorem*{propRD_II}{\cite{R2009} - Proposition 2.1}

\renewcommand\thefootnote{\fnsymbol{footnote}}

\newcommand\f[2]{\frac{#1}{#2}}
\newcommand\lb{\left(}
\newcommand\rb{\right)}
\newcommand\ignore[1]{}
\newcommand\R{\mathbb R}
\newcommand\M{\mathbb M}
\newcommand\N{\mathbb N}

\newcommand\E[2][\P]{\mathbb E_{#1}\left[ #2\right]}
\renewcommand\E{\mathbb E}
\renewcommand\P{\mathbb P}

\newcommand\ind{1\hspace{-2.5mm}1}

\newcommand\ldju[2]{\bigcup_{#1}^{#2}\hspace{-10.65mm\cdot}\hspace{8mm}}
\newcommand\cdju[2]{\bigcup_{#1}^{#2}\hspace{-5.0mm\cdot}\hspace{2mm}}

\newcommand\fadd[1]{{\color{red}#1}}

\makeatletter
\renewcommand{\@seccntformat}[1]{\csname the#1\endcsname.\hspace{1em}}%

\makeatother

\begin{document}

% \title{\it{...}}
% \date{}
% \maketitle

{\centering {\renewcommand\thefootnote{}\large\bfseries An analysis of the R\"{u}schendorf transform - with a view towards Sklar's Theorem
% revisited: an elaboration of the R\"{u}schendorf transform approach
%\footnote{The authors gratefully
%acknowledge support from EPSRC grant no. GR/S80202/01.}
}

\par

\vspace{1em} {\scshape Frank Oertel\footnote{Deloitte \& Touche GmbH, FSI Assurance, Quantitative Services \& Valuation, D - 81669 Munich, E-mail: f.oertel@email.de}}
%\par\itshape Deloitte \& Touche GmbH\par FSI Assurance \par Quantitative Services \& Valuation \par D - 81669 Munich

%\par \vspace{1em} {\scshape
% Mark P. Owen}
% \par\itshape Department of Actuarial Mathematics and Statistics\par
% Maxwell Institute for Mathematical Sciences and
% Heriot-Watt University
% Department of Actuarial Mathematics and Statistics\\ Heriot-Watt University
}
\vspace{1cm}
%\begin{abstract}\renewcommand\thefootnote{}\footnotesize
\noindent {\bf{Abstract:}} In many applications including financial risk measurement, copulas have shown to be a powerful building block to reflect multivariate dependence between several random variables including the mapping of tail dependencies.
%\\[0.5em]

A famous key result in this field is Sklar's Theorem. Meanwhile, there exist several approaches to prove Sklar's Theorem in its full generality.
%(even under inclusion of mathematically advanced results from topology and functional analysis). 
An elegant probabilistic proof was provided by L. R\"{u}schendorf. To this end he implemented a certain ``distributional transform'' which naturally transforms an arbitrary distribution function $F$ to a flexible parameter-dependent function
%between $F-$ and $F$,
which exhibits exactly the same jump size as $F$.
%lying between the left-limit $F-$ and $F$ itself.  
%\\[0.5em]

By using some real analysis and measure theory only (without involving the use of a given probability measure) we expand into the underlying rich structure of the distributional transform. 
%In particular, we show that for a large class of non-decreasing and right-continuous functions with jumps the (algebraic) left-inverse function of a function $F$ belonging to this class is ``almost'' given by $F^{\wedge}(\alpha) : = \inf\{x \in {\R} : F(x) \geq \alpha \}$ (Theorem \ref{thm:inversion_almost_surely}).
%\\[0.5em]
Based on derived results from this analysis (such as Proposition \ref{thm:in_the_quantiles_range} and Theorem \ref{thm:inversion_almost_surely}) including a strong and frequent use of the \textit{right} quantile function, we revisit R\"{u}schendorf's proof of Sklar's theorem and provide some supplementing observations including a further characterisation of distribution functions (Remark \ref{thm:close_to_a_df}) and a strict mathematical description of their ``flat pieces'' (Corollary \ref{thm:description_of_flat_pieces_in_general} and Remark \ref{BCBS_and_newsvendor_problem}).
\\[1em]
\\[1em] 
%\end{abstract}

%\section{A detailed workout of a key result of R\"{u}schendorf}
\section{Introduction}\label{intro}
The mathematical investigation of copulas started 1951, due to the following problem of M. Fr\'{e}chet: suppose, one is given $n$ random variables $X_1, X_2, \ldots, X_n$, all defined on the same probability space $(\Omega, \mathcal{F}, \P)$, such that each random variable has a (non-necessarily continuous) distribution function $F_i$ $(i = 1, 2, \ldots, n)$. What can then be said about the set of all possible $n$-dimensional distribution functions of the random vector $(X_1, X_2, \ldots, X_n)$ (cf. \cite{F1951})? This question has an immediate answer if the random variables were assumed to be independent, since in this case there exists a unique $n$-dimensional distribution function of the random vector $(X_1, X_2, \ldots, X_n)$, which is given by the product $\Pi_{i=1}^n F_i$. However, if the random variables are not independent, there was no clear answer to M. Fr\'{e}chet's problem.
%\\[1em]

In \cite{S1959}, A. Sklar introduced the expression ``copula'' (referring to a grammatical term for a word that links a subject and predicate), and provided answers to some of the questions of M. Fr\'{e}chet.
%\\[1em]

In the following couple of decades, copulas (which are precisely finite dimensional distribution functions with uniformly distributed marginals), were mainly used in the framework of probabilistic metric spaces (cf. e.\,g. \cite{SchS1974, SchS1983}). Later, probabilists and statisticians were interested in copulas, since copulas defined in a ``natural way'' nonparametric measures of dependence between random variables, allowing to include a mapping of tail dependencies. Since then, they began to play an important role in several areas of probability and statistics (including Markov processes and non-parametric statistics), in financial and actuarial mathematics (particularly with respect to the measurement of credit risk), and even in medicine and engineering.
%\\[1em]

One of the key results in the theory and applications of copulas, is Sklar's Theorem (which actually was proven in \cite{SchS1974} and not in \cite{S1959}). It says:
\begin{Sklar}%[Sklar]
%{\tt{Theorem (Sklar)}}\newline
Let $F$ be a $n$-dimensional distribution function with marginals $F_1, \ldots, F_n$. Then there exists a copula $C_F$, such that for all $(x_1, \ldots, x_n) \in {\R}^n$ we have
\[
F(x_1, \ldots, x_n) = C_F(F_1(x_1), \ldots, F_n(x_n))\,.
\]
Furthermore, if $F$ is continuous, the copula $C_F$ is unique. Conversely, for any univariate distribution functions $H_1, \ldots, H_n$, and any copula $C$, the composition $C \circ (H_1, \ldots, H_n)$ defines a $n$-dimensional distribution function with marginals $H_1, \ldots, H_n$.
\end{Sklar}
\noindent Since the original proof of (the general non-continuous case of) Sklar's Theorem is rather complicated and technical, there have been several attempts to provide different and more lucidly appearing proofs, involving not only techniques from probability theory and statistics but also from topology and functional analysis (cf. \cite{DFsS2013}).
%\\[1em]

Among those different proofs of Sklar's Theorem, there is an elegant, yet rather short proof, provided by L. R\"{u}schendorf, originally published in \cite{R2009}. He provided a very intuitive, and primarily probabilistic approach which allows to treat general distribution functions (including discrete parts and jumps) in a similar way as continuous distribution functions. To this end, he applied a generalised ``distributional transform'' which - according to \cite{R2009} - has been used in statistics for a long time in relation to a construction of randomised tests. By making a consequent use of the properties of this generalised ``distributional transform'' \textit{together} with Proposition 2.1 in \cite{R2009}, the proof of Sklar's Theorem in fact follows immediately (cf. Theorem 2.2 in \cite{R2009}). Irrespectively of \cite{R2009} the same idea was used in the (again rather short) proof of Lemma 3.2 in \cite{MS1975}. All key inputs for the proof of Sklar's Theorem clearly are provided by Proposition 2.1 in \cite{R2009}. However, the proof of the latter result is rather difficult to reconstruct. It says:
\begin{propRD_II}\label{thm:Prop_2_1}
Let $X, V$ be two random variables, defined on the same probability space $(\Omega, {\mathcal{F}}, \P)$, such that $V \sim U(0,1)$ and $V$ is independent of $X$. Let $F$ be the distribution function of the random variable $X$. Then $U : = F_V(X) \sim U(0,1)$, and $X = F^{-}(U)$ $\P$-almost surely.
\end{propRD_II}
\noindent Here,
%\[
$F^{-}(\alpha) = \inf\{x \in {\R} : F(x) \geq \alpha \}$
%= \sup \{x \in {\R} : F(x) < u \}
%\hspace{0.5cm} (0 < \alpha < 1)
%\]
denotes the (left-continuous) left $\alpha$-quantile of $F$ which in particular is the lowest generalised inverse of $F$ (cf. e.g. \cite[Chapter 4.4]{SchS1983}, respectively \cite[Definition 2]{KMP1999}).
%\cite{SchS1983}, Chapter 4.4; \cite{KMP1999}, Definition 2 ).
%(cf. e.g. Chapter 4.4 in \cite{SchS2006}, Definition 2 in \cite{KMP1999}, % Definition A.20. in \cite{FS2011} and Definition 2.1 in \cite{EH2013}). 
In our paper we consistently adopt the very suitable symbolic notation of \cite{SchS1983}, respectively \cite{KMP1999} to identify generalised inverse functions in general (cf. \eqref{eqn:smallest_gen_inverse} and \eqref{eqn:gen_inverses}).
%\\[1em]

While studying (and reconstructing) carefully the proof of Sklar's Theorem built on Proposition 2.1 in \cite{R2009}, we recognise that it actually implements key mathematical objects which do not involve probability theory at all and play an important role beyond statistical applications.
%\\[1em]

The main contribution of our paper is to provide a thorough analysis of these mathematical building blocks by studying carefully properties of a real-valued (deterministic) function, used in the proof of Proposition 2.1 in \cite{R2009}; the so-called ``R\"{u}schendorf transform''. We reveal some interesting structural properties of this function which to the best of our knowledge have not been published before, such as e.\,g. Theorem \ref{thm:inversion_almost_surely} which actually is a result on Lebesgue-Stieltjes measures, strongly built on the role of the \textit{right} quantile function which seems to be not widely used in the literature (as opposed to the left quantile function). 
%does not involve randomness at all.
%\\[1em]

Equipped with Theorem \ref{thm:inversion_almost_surely} we then revisit the proof of Proposition 2.1 in \cite{R2009} (cf. also \cite[Chapter 1.1.2]{MS2012}). However, in our approach Proposition 2.1 in \cite{R2009} \textit{is an implication} of Theorem \ref{thm:inversion_almost_surely} and Lemma \ref{thm:prob_calc_and_rv_as_gen_inverse_value}. For sake of completeness we include a proof of Sklar's Theorem again (cf. also \cite[Chapter 1.1.2]{MS2012}) - yet as an implication of Theorem \ref{thm:inversion_almost_surely}, finally leading to Remark \ref{thm:obsv_on_Sklars_thm}.
%\\[1em]

Last but not least, by observing the significance of the jumps of the lowest generalised inverse, the proof of Theorem \ref{thm:inversion_almost_surely} indicates how to construct the $\P$-null set in Proposition 2.1 in \cite{R2009} explicitly - leading to Theorem \ref{thm:Rueschendorf}.
\section{The R\"{u}schendorf Transform}\label{RT}
At the moment let us completely ignore randomness and probability theory. We ``only'' are working within a subclass of real-valued functions, all defined on the real line, and with suitable subsets of the real line. 
%\\[1em]

Let $F : \R \longrightarrow \R$ be an arbitrary right-continuous and non-decreasing function. Let $x \in \R$. Since $F$ is non-decreasing, it is well-known that both, the left-hand limit
\[
F(x-) : = \lim_{z \uparrow x} F(z) = \sup\left\{F(z) : z \leq x \right\}\,,
% = \lim_{n \to \infty} F\lb x - \f{1}{n}\rb % = \sup\left\{F\lb x - \f{1}{n}\rb: n\in \N\right\}
\]
and the right-hand limit
\[
F(x+) : = \lim_{z \downarrow x} F(z) = \inf\left\{F(z) : z \geq x \right\}
%= \lim_{n \to \infty} F\lb x + \f{1}{n}\rb % = \sup\left\{F\lb x - \f{1}{n}\rb: n\in \N\right\}
\]
are well-defined real numbers, satisfying $F(x-) \leq F(x) \leq F(x+)$. Moreover, due to the assumed right-continuity of $F$, it follows that $F(x) = F(x+)$ for all $x \in \R$. $0 \leq \Delta F(x) : = F(x+) - F(x-) = F(x) - F(x-)$ denotes the (left-hand) ``jump'' of $F$ at $x$. We consider the following important transform of $F$:
\begin{definition}%[R\"{u}schendorf Transform]
Let $\lambda \in [0,1]$ and $x \in \R$. Put
\[
R_F(x, \lambda) : = F_\lambda(x) = F(x-) + \lambda \Delta F(x)\,.
\]
We call the real-valued function $R_F : \R \times [0, 1] \longrightarrow \R$ the R\"{u}schendorf transform of $F$. For given $\lambda \in [0,1]$ $F_\lambda : \R \longrightarrow \R$ is called the R\"{u}schendorf $\lambda$-transform of $F$.
\end{definition}
\noindent Clearly, we have the following equivalent representation of the R\"{u}schendorf $\lambda$-transform $F_\lambda$:
\[
F_\lambda(x) = (1 - \lambda)F(x-) + \lambda F(x) \mbox{ for all } x \in \R\,.
\]
In particular, for all $(x, \lambda) \in \R \times [0,1]$ the following inequality holds:
\begin{equation}\label{eq:quantile_estimation}
F(x-) \leq F_\lambda(x) \leq F(x)\,.
%\footnote{Hence, if $F$ were the distribution function of a random variable $X$, $x$ would be precisely the $F_\lambda(x)$-quantile of $X$.}
\end{equation}
Moreover, $F$ is continuous if and only if $F(x-) = F_\lambda(x)$ for all $(x, \lambda) \in \R \times (0,1]$, and for all $(x, \lambda) \in \R \times [0,1]$ we have $F_0(x) = F(x-) = F_\lambda(x-)$ and $F_1(x) = F(x) = F_\lambda(x+)$.
%Note also that $F_\lambda(x) = F(x-)$ if and only if $F$ is continuous at $x$ (since $\lambda > 0$).
%$if and only if $\Delta F(x) = 0$ (i.\,e, if and only if $F$ is continuous at $x$).
%\\[0.5em]
\begin{assumption}\label{thm:assumption_on_F}
In the following we assume throughout that $F$ is bounded on $\R$ (i.\,e., the range $F(\R)$ is a bounded subset of $\R$), implying that $F(\R) \subseteq [c_\ast, c^\ast]$ for some real numbers $c_\ast < c^\ast$. 
%Since $F$ is non-decreasing, we may assume without loss of generality that $c_\ast = \inf \big(F(\R)\big) \stackrel{(???)}{=}  \lim\limits_{n\to\infty}F(-n)$
% and $c^\ast = \sup\big(F(\R)\big) \stackrel{(???)}{=} \lim\limits_{n\to\infty}F(n)$.
Moreover, let us assume that for any $\alpha \in (c_\ast, c^\ast)$ the set $\{x \in \R : F(x) \geq \alpha\}$ is non-empty and bounded from below.\footnote{In particular, $F$ cannot be a constant function on the whole real line.} WLOG, we may assume from now on that $c_\ast = 0$ and $c^\ast = 1$ (else we would have to work with the function $\frac{F-c_\ast}{c^\ast - c_\ast}$).
\end{assumption}
Although its proof (by contradiction) mostly is an easy calculus exercise with sequences, the following observation - which does not require a right-continuity assumption - should be explicitly noted (cf. also (cf. \cite{EH2013, FS2011, SchS1974})):
\begin{remark}\label{thm:close_to_a_df}
Let $G : \R \longrightarrow [0,1]$ an arbitrary non-decreasing function. Then the following statements are equivalent:
\begin{itemize}
\item[(i)] $\lim\limits_{x \to -\infty}G(x) = 0$ and $\lim\limits_{x \to \infty}G(x) = 1$;
\item[(ii)] For any $\alpha \in (0, 1)$ the sets $\{x \in \R : G(x) < \alpha\}$ and $\{x \in \R : G(x) \geq \alpha\}$ both are non-empty;
\item[(iii)] For any $\alpha \in (0, 1)$ the set $\{x \in \R : G(x) \geq \alpha\}$ is non-empty and bounded from below.
\item[(iv)] $G^{\wedge}(\alpha) : = \inf\{x \in {\R} : G(x) \geq \alpha \}$ is a well-defined real number for any $\alpha \in (0, 1)$. 
\end{itemize}
\end{remark}
Hence, given Assumption \ref{thm:assumption_on_F}, the assumed right-continuity of $F$ and Remark \ref{thm:close_to_a_df} imply that (possibly after shifting and stretching $F$ adequately) $F$ actually is a distribution function! Therefore, its generalised inverse function $F^{\wedge} : (0, 1) \longrightarrow \R$, given by
% (by adapting the suggestive notation of \cite{EH2013}, respectively \cite{FS2011})
%\footnote{The attached copy is an excerpt from the appendix of \cite{FS2011}.}
\begin{equation}\label{eqn:smallest_gen_inverse}
F^{\wedge}(\alpha) : = \inf\{x \in {\R} : F(x) \geq \alpha \},
%= \sup \{x \in {\R} : F(x) < \alpha \}
\end{equation}
is well-defined and satisfies
\begin{equation}\label{eqn:gen_inverses}
-\infty < F^{\wedge}(\alpha) \leq F^{\wedge}(\alpha+) = \inf\{x \in {\R} : F(x) > \alpha \} = \sup\{x \in {\R} : F(x) \leq \alpha \} = : F^{\vee}(\alpha) < \infty
\end{equation}
for any $\alpha \in (0, 1)$ (cf. e.\,g. \cite{FKTW2012}).
%\footnote{If $F : \R \longrightarrow [0,1]$ were a distribution function, this assumption would be satisfied \fadd{since $\lim\limits_{x \to -\infty} F(x) = 0 = 0$ and $\lim\limits_{x \to \infty} F(x) = 1= 1$}, implying that in fact $\vert F^{\wedge}(\alpha) \vert < \infty$ for all $\alpha \in (0, 1)$.}
Actually, since $F$ is assumed to be right-continuous, it follows that
\[
F^{\wedge}(\alpha) = \min\{x \in {\R} : F(x) \geq \alpha \}
\]
for all $\alpha \in (0, 1)$ (cf. \cite[Proposition 2.3\,(4)]{EH2013}).
% Let $\alpha \in (0, 1)$ such that $F^{\wedge}(\alpha) \in \R$.
Moreover, the following important inequality is satisfied:
\begin{equation}
\label{eqn:quantile_strict_ineq}
F \lb F^{\wedge}(\alpha) - \delta \rb < \alpha \leq F\lb F^{\wedge}(\alpha) + \varepsilon \rb\,
\end{equation}
for all $\alpha \in (0, 1)$, $\delta > 0$, and for all $\varepsilon > 0$. Hence,
\begin{equation}\label{eqn:quantile_ineq}
F\lb F^{\wedge}(\alpha)- \rb \leq \alpha \leq F\lb F^{\wedge}(\alpha)+\rb = F\lb F^{\wedge}(\alpha)\rb
\end{equation}
for all $\alpha \in (0, 1)$. Also recall from e.\,g. \cite{SchS1983} that $\{x \in \R : F(x) \geq \alpha\} = [F^{\wedge}(\alpha), \infty)$, respectively $\{x \in \R : F(x) < \alpha\} = (-\infty, F^{\wedge}(\alpha))$ for any $\alpha \in (0, 1)$.
%\\[0.5em]

Let us fix the distribution function $F : \R \longrightarrow [0, 1]$. Then by $J_F : = \{x \in \R : \Delta F(x) > 0\}$ we denote the set of all jumps of $F$ which is well-known to be at most countable.

Throughout the remaining part of our paper, we follow the notation of \cite{R2009} and put $\xi : = F^{\wedge}(\alpha)$ for fixed $0 < \alpha < 1$. By taking a closer look at $F^{\wedge}\big(F_\lambda(x)\big)$, we firstly note the following observation.
\begin{remark}\label{thm:inversion_inequality_on_Omega}
Let $\lambda \in [0,1]$ and $0 < F_\lambda(x) < 1$. Then
\[
F^{\wedge}\big(F_\lambda(x)\big) \leq x\,.
\]
\end{remark}
\begin{proof}
% Fix $0 \leq \lambda \leq 1$ and $x \in \R$ such that $0 < F_\lambda(x) < 1$. 
Fix $\lambda \in [0,1]$ and put $\alpha : = F_\lambda(x)$, where $x \in F_\lambda^{-1}((0, 1))$. Then $F^{\wedge}\big(\alpha\big)$ is well-defined. Since $F(x) \geq F_\lambda(x) = \alpha$, the claim follows.
\end{proof}
The next result shows an important part of the role of R\"{u}schendorf transform which can be more easily understood if one sketches the graph of $F$ including its jumps. Since $J_F$ is at most countable, it follows that $J_F = \{x_n : n\in M\}$, where either $card(M) < \infty$ or $M = \N$. By making use of this representation and the canonically defined function $F^{-} : \R \longrightarrow [0,1]$, $x \mapsto F(x-)$ (cf. also \cite[Chapter 4.4]{SchS1983}) we arrive at the following
\begin{proposition}\label{thm:in_the_quantiles_range}
\begin{itemize}
%\item[(i)]
Let $x \in \R$. Then
\[
\big(F(x-), F(x)\big) \subseteq \big\{\alpha \in (0, 1) : x = F^{\wedge}(\alpha)\big\} \subseteq \big[F(x-), F(x)\big]\,.
\]
In particular, if $x_1 \not = x_2$ then $\big(F(x_1-), F(x_1)\big) \cap \big(F(x_2-), F(x_2)\big) = \emptyset$.
%\[
%F^{\wedge}\big(F_\lambda(x)\big) = x
%\]
%if $\Delta F(x) > 0$ and $0 < \lambda < 1$. 
Moreover,   
\begin{eqnarray*}
\cdju{x \in J_F}{}\big(F(x-), F(x)\big) & = & \big\{\alpha \in (0, 1): \Delta F(F^\wedge(\alpha))>0 \text{ and } \alpha = F_\lambda(F^\wedge(\alpha)) \text{ for some } 0 < \lambda < 1\big\}\\
& =  & \big\{F_\lambda(x) : 0 < \lambda < 1 \text{ and } x \in J_F\big\}\\% \, = \, R_F\big((0,1) \times J_F\big)\\
& = & R_F\big(J_F \times (0,1)\big)\\
& = & (0, 1)\setminus (F(\R) \cup F^{-}(\R))\,,
\end{eqnarray*}
implying that 
%if $x^\ast \in J_F \not= \emptyset$, 
the mapping $\Phi_F : (0,1)^{M} \longrightarrow \prod_{n=1}^{card(M)}\big(F(x_n-), F(x_n)\big)$, $\big(\lambda_n\big)_{n \in M} \mapsto \big({F_{\lambda_n}}(x_n)\big)_{n \in M}$ is well-defined and bijective. Its inverse $\Phi_F^{-1} : \prod_{n=1}^{card(M)}\big(F(x_n-), F(x_n)\big) \longrightarrow (0,1)^{M}$ is given by
\[
\Phi_F^{-1}\big(\big(\alpha_n\big)_{n \in M}\big) = \Big(\frac{\alpha_n - F(F^\wedge(\alpha_n)-)}{\Delta F(F^\wedge(\alpha_n))}\Big)_{n \in M}\,. 
%\hspace{1cm} \Big((\alpha_n \in \cdju{x \in J_F}{}\big(F(x-), F(x)\big)\Big)\,.
\]
%Finally,
%\[
%\cdju{x \in J_F}{}\big(F(x-), F(x)\big) = (0, 1)\setminus (F(\R) \cup F^{-}(\R))\,,
%\]
%where    
%where $J_F : = \{x \in \R : \Delta F(x) > 0\}$ denotes the set of all jumps of $F$.
\end{itemize}
\end{proposition}
\begin{proof}
To prove the first set inclusion, we may assume without loss of generality that $F$ is not continuous in $x$. So, let $F(x-) < \alpha < F(x)$. Then $0 < \alpha < 1$ (else we would obtain the contradiction $\alpha \leq 0 \leq F(x-)$, respectively $F(x) \leq 1 \leq \alpha$) and $F(x - \frac{1}{n}) < \alpha \leq F(x)$ for all $n \in \N$. Hence, $F^{\wedge}(\alpha) \leq x < F^{\wedge}(\alpha) + \frac{1}{n}$ for all $n \in \N$ (cf. \cite[Proposition 2.3\,(5)]{EH2013}), implying the first inclusion. Now let $\alpha \in (0, 1)$ such that $x = F^{\wedge}(\alpha)$. Due to \eqref{eqn:quantile_ineq} it follows that
\[
F\lb x- \rb \leq \alpha \leq F\lb x \rb\,,
\]
which gives the second set inclusion.

To verify the representation of the disjoint union $\ldju{x \in J_F}{}\big(F(x-), F(x)\big)$ let $\alpha \in \big(F(x-), F(x)\big)$ for some $x \in J_F$. Then $x = F^{\wedge}(\alpha) = : \xi$ and hence $\Delta F(\xi) > 0$ and $\alpha \in \big(F(\xi-), F(\xi)\big)$. Put 
\[
\lambda(\alpha) : = \frac{\alpha - F(\xi-)}{\Delta F(\xi)}\,.
\]
Then $0 < \lambda(\alpha) < 1$ and 
\[
\alpha = F_{\lambda(\alpha)}(\xi) = F_{\lambda(\alpha)}(x)\,.
\]
Furthermore, a straightforward application of the inequality \eqref{eqn:quantile_strict_ineq} (together with \eqref{eqn:quantile_ineq} and the monotonicity assumption on $F$) shows the graphically clear fact that there is no $x \in J_F$ such that $(F(x-), F(x))$ contains elements of the form $F(z)$, respectively $F(w-)$ for some $z,w \in \R$.
Now, given the construction of $\lambda(\alpha)$ above and the listed properties of any of the sets $\big(F(x_n-), F(x_n)\big)$, the assertion about the mapping $\Phi_F$ follows immediately.
\end{proof}
%\[
%\big(F(x-), F(x)\big) \subseteq \big\{\alpha \in (0, 1) : x =
%F^{\wedge}(\alpha)\big\}\,.
%\]
\begin{definition}
Let $\alpha \in (0, 1)$ and $\lambda \in [0,1]$. Put:
\[
A_{\lambda, \alpha} : = \left\{x \in \R :  F_\lambda(x) \leq \alpha \right\}\,.
\]
\end{definition}
%At this point, to facilitate reading, we ignore the explicit dependence of the function $F$ if $F$ is specified in advance. So, let us fix $F$.
Firstly note that $A_{\lambda, \alpha}$ is non-empty. To see this, consider any $x < \xi = F^{\wedge}(\alpha)$. Then $x \leq \xi - \delta$ for some $\delta > 0$. Hence, $F_\lambda(x) \leq F(x) \stackrel{\eqref{eqn:quantile_strict_ineq}}{<} \alpha$. To motivate the following representation of the set $A_{\lambda, \alpha}$, let us assume for the moment that $F$ is continuous at $\xi$. Due to \eqref{eqn:quantile_ineq}, it follows that $F(\xi) = \alpha$. Hence, in this case, $F_\lambda(\xi) = F(\xi) = \alpha$, implying that $\xi = F^{\wedge}(\alpha) \in A_{\lambda, \alpha}$.
%\\[1em]

However, in the general (non-continuous) case, $\xi = F^{\wedge}(\alpha) $ need not be an element of the set $A_{\lambda, \alpha}$. Therefore (by fixing $\alpha \in (0, 1)$ and $\lambda \in [0,1]$), we are going to represent the set $A_{\lambda, \alpha}$ as a disjoint union of the following three subsets of the real line:
\[
A_{\lambda, \alpha}^{+} : = A_{\lambda, \alpha} \cap \{x \in \R :  x > \xi\}\,,
\]
\[
A_{\lambda, \alpha}^{\sim} : = A_{\lambda, \alpha} \cap \{x \in \R :  x = \xi\}\,,
\]
and
\[
A_{\lambda, \alpha}^{-} : = A_{\lambda, \alpha} \cap \{x \in \R :  x < \xi\}\,.
\]
Thus,
\[
A_{\lambda, \alpha} = A_{\lambda, \alpha}^{+} \dotcup A_{\lambda, \alpha}^{\sim} \dotcup A_{\lambda, \alpha}^{-}\,.
\]
Next, we are going to simplify the sets $A_{\lambda, \alpha}^{+}, A_{\lambda, \alpha}^{\sim}$ and $A_{\lambda, \alpha}^{-}$ as far as possible. To this end, we have to analyse carefully the jump $\Delta F^{\wedge}(\alpha)$, implying that we have to check $\xi = F^{\wedge}(\alpha) = F^{\wedge}(\alpha-)$ against the (finite) value of the \textit{largest generalised inverse of $F$} (cf. \cite{FKTW2012} and \cite[Chapter 4.4]{SchS1983})%\footnote{In \cite{SchS1983}, $F^{\vee}$ is called \textit{...right-continuous inverse function of $F$}.}
\[
\eta : = F^{\vee}(\alpha) = \inf\{x \in \R : F(x) > \alpha\} = \sup\{x \in \R : F(x) \leq \alpha\} = F^{\wedge}(\alpha+)\,.
\]
The inequality \eqref{eqn:quantile_ineq} is also satisfied for $\eta$ (cf. \cite[Lemma A.\,15]{FS2011}):
\begin{equation}\label{eqn:quantile_ineq_for_eta}
F\lb \eta- \rb \leq \alpha \leq F\lb \eta\rb\,.
\end{equation}
%for all $\alpha \in (0, 1)$.
Note that since $F$ is a distribution function, $\eta$ (respectively $\xi$) is precisely the right (respectively left) $\alpha$-quantile of $F$.
%\\[0.5em]

Clearly, $\{x \in \R :  x > \xi \mbox{ and } F(x) = \alpha\} \subseteq {A_{\lambda, \alpha}^{+}}$ for every $\lambda \in [0,1]$. However, if $0 < \lambda \leq 1$, we even obtain equality of both sets - since:
%%% CHECKED UNTIL HERE  (16/01/2015)
%(cf. \cite{FS2011}).
\begin{lemma}\label{thm:flat_pieces_of_F}
Let $0 < \lambda \leq 1$ and $\alpha \in (0, 1)$. Put $\xi : = F^{\wedge}(\alpha)$ and $\eta : = F^{\vee}(\alpha)$.
\begin{itemize}
\item[(i)] If $\xi < \eta$, then $F(\xi) = \alpha = F(\eta-) \notin \ldju{x \in J_F}{}\big(F(x-), F(x)\big)$ and $\emptyset \not= \{x \in \R :  x > \xi \mbox{ and } F(x) = \alpha\}$. Moreover, the restricted function $F\vert_{A_{\lambda, \alpha}^{+}} : A_{\lambda, \alpha}^{+} \longrightarrow \R$ is continuous, and
\begin{equation}\label{eqn:flat_pieces_of_F_Part_1}
A_{\lambda, \alpha}^{+} = \{x \in \R :  x > \xi \mbox{ and } F(x) = \alpha\} = \begin{cases}\big(\xi, \eta\big)&\text{if
  }F(\eta) > \alpha\\ \big(\xi, \eta\big]&\text{if }F(\eta) = \alpha\end{cases}
\end{equation}
\item[(ii)] If $\xi = \eta$, then $A_{\lambda, \alpha}^{+} = \emptyset$.
\item[(iii)] Furthermore,
\begin{eqnarray*}
%\R \setminus F^\wedge((0, 1)) \cup \R \setminus F^\vee((0, 1)) \subseteq \{x \in \R : F \text{ is continuous in } x\}\,.
J_F & = & \big\{x \in \R :  F^{\wedge}(u) = x = F^{\vee}(u) \text{ and } \Delta F(F^{\wedge}(u)) > 0 \text{ for some } u \in (0, 1)\big\}\\
& = & \big\{x \in \R :  F^{\wedge}(u) = x = F^{\vee}(u) \text{ and } \Delta F(F^{\vee}(u)) > 0 \text{ for some } u \in (0, 1)\big\}\,.
%& = & J_F^{+} \cup J_F^{-}\,,
\end{eqnarray*}
\end{itemize}
In particular, the following statements are equivalent:
\begin{itemize}
\item[(a)] $0 < \Delta F^{\wedge}(\alpha) = \eta - \xi$;
%\item[(b)] $A_{\lambda, \alpha}^{+} \not= \emptyset$;
\item[(b)] $\{x \in \R :  x > \xi \mbox{ and } F(x) = \alpha\} \not= \emptyset$.
\end{itemize}
\end{lemma}
\begin{proof}
Put $B : = \{x \in \R :  x > \xi \mbox{ and } F(x) = \alpha\}$. Clearly, we always have $B \subseteq A_{\lambda, \alpha}^{+}$.
%\\[0.5em]

To verify (i), let $\xi < \eta$. Then $\xi < z_0 < \eta = \inf\{x \in \R : F(x) > \alpha\}$ for some $z_0 \in \R$. Thus, $F(\xi) \leq F(z_0) \leq \alpha \leq F(\xi)$, implying that $z_0 \in B$ and $F(\xi) = \alpha$. Assume by contradiction that $F(\eta-) < \alpha$. Then $F(\eta - \varepsilon) < F(\xi)$ for all $\varepsilon > 0$, implying the contradiction $\eta \leq \xi$. Hence, $F(\eta-) = \alpha$. Proposition \ref{thm:in_the_quantiles_range} therefore implies that $\alpha = F(\xi) \notin \ldju{x \in J_F}{}\big(F(x-), F(x)\big)$. 

Let $x \in A_{\lambda, \alpha}^{+} \supseteq B$. Assume by contradiction that $F\vert_{A_{\lambda, \alpha}^{+}}$ is not continuous at $x$. Then $F(x-) < F(x-) + \lambda \Delta F(x) = F_\lambda(x) \leq \alpha$ (since $\lambda >0$). Since $x > \xi$, we have $\xi \leq x - \f{1}{n}$ for some $n \in \N$. Thus,
\[
\alpha \leq F(\xi) \leq F\Big(\xi + \f{1}{2n}\Big) \leq F\Big(x - \f{1}{2n}\Big)\,.
\]
Hence, $\alpha \leq F(x-) < \alpha$, which is a contradiction. Thus, the restricted function $F\vert_{A_{\lambda, \alpha}^{+}}$ is continuous on $A_{\lambda, \alpha}^{+}$. Let $u \in A_{\lambda, \alpha}^{+}$. Since $F$ is continuous at $u$, it follows that
\[
\alpha \leq F(\xi) \leq F(u) = F_\lambda(u) \leq \alpha\,.
\]
Thus, $\emptyset \not= A_{\lambda, \alpha}^{+} = B$.
%\\[0.5em]

To prove (ii), suppose that $A_{\lambda, \alpha}^{+}$ is non-empty. The previous calculations show that the existence of an element $u_0 \in A_{\lambda, \alpha}^{+}$ already implies $F(\xi) = F(u_0) = \alpha$.
% for some $u_0 \in A_{\lambda, \alpha}^{+}$.
Consequently, $\eta = F^{\vee}(\alpha) = \sup\{x \in \R : F(x) \leq \alpha \}$ cannot coincide with $\xi = F^{\wedge}(\alpha)$ (since $\xi < u_0 \leq \eta$), implying that $\xi < \eta$.
%\\[0.5em]

To finish the proof of (i), we have to verify \eqref{eqn:flat_pieces_of_F_Part_1}. To this end, let $\xi < \eta$ and $x \in (\xi, \eta)$. Then there exists $\delta > 0$ such that $\xi < x-\delta < x < x+\delta <  \eta = F^{\vee}(\alpha) = \inf\{u \in \R : F(u) > \alpha\}$. Consequently, $\alpha \leq F(\xi) \leq F(x-\delta) \leq F(x) \leq F(x+\delta) \leq \alpha$. Thus,
\[
\big(\xi, \eta\big) \subseteq \{x \in \R :  x > \xi \mbox{ and } F(x) = \alpha\} = B\,.% = A_{\lambda, \alpha}^{+}\,.
\]
%Let $x \in A_{\lambda, \gamma}^{+}$. Then $x > \xi$ and $F_\lambda(x) \leq \alpha$, implying that the non-empty open interval $\big] \frac{3\xi + x}{4}, \frac{\xi+3x}{4}\big[$ is well-defined and contained in the non-empty open interval $]\xi, x[$ which itself is a subset of $A_{\lambda, \alpha}^{+}$.
Moreover, \cite[Proposition 2.3\,(6)]{EH2013} implies that
\[
%A_{\lambda, \alpha}^{+} =
B = \{x \in \R :  x > \xi \mbox{ and } F(x) = \alpha\} \subseteq \big(\xi, \eta \big]\,.
\]
Hence,
\[
\big(\xi, \eta \big) \subseteq B \subseteq \big(\xi, \eta \big]\,.
\]
If $F(\eta) > \alpha$, then $\eta \notin B$ and hence $B = \big(\xi, \eta \big)$. If $F(\eta) = \alpha$, then $\xi < \eta \in B$ and hence $B = \big(\xi, \eta \big]$.

Statement (iii) is a direct implication of (i) and Proposition \ref{thm:in_the_quantiles_range}.
\end{proof}
Regarding a visualisation of Lemma \ref{thm:flat_pieces_of_F} consider the set $M_\alpha : = \{x \in \R :  x \leq \xi \mbox{ and } F(x) = \alpha\} \stackrel{\eqref{eqn:quantile_strict_ineq}}{=} \{x \in \R :  x = \xi \mbox{ and } F(x) = \alpha\} \in \big\{\emptyset, \{\xi\}\big\}$. Note that
\[
\{x \in \R : F(x) = \alpha\} = \{x \in \R :  x > \xi \mbox{ and } F(x) = \alpha\} \dotcup M_\alpha\,.
\]
Thus, by joining Lemma \ref{thm:flat_pieces_of_F} with Proposition \ref{thm:in_the_quantiles_range} we immediately obtain the following tangible mathematical description of the (preimages of) ``flat pieces'' of $F$ (and hence allowing us to perfect related observations from e.\,g.\cite[Chapter 4.4]{SchS1983} and \cite{EH2013}, Proposition 2.3, (6) coherently):
\begin{corollary}\label{thm:description_of_flat_pieces_in_general}
Let $0 < \lambda \leq 1$ and $\alpha \in (0, 1)$. Put $\xi : = F^{\wedge}(\alpha)$ and $\eta : = F^{\vee}(\alpha)$.
\begin{itemize}
\item[(i)] If $\xi < \eta$, then
\begin{equation*}%\label{eqn:flat_pieces_of_F_Part_2}
\emptyset \not= \{x \in \R : F(x) = \alpha\} = A_{\lambda, \alpha}^{+} \dotcup \{\xi\} = \begin{cases}%\big(\xi, \eta\big)&\text{if }F(\xi) > \alpha\\
\big[\xi, \eta\big)&\text{if }F(\eta) > \alpha\\ \big[\xi, \eta\big]&\text{if }F(\eta) = \alpha\end{cases}
\end{equation*}
%In particular, $F(\xi) = \alpha$.
\item[(ii)] If $\xi = \eta$, then
\begin{equation*}%\label{eqn:flat_pieces_of_F_Part_3}
\{x \in \R : F(x) = \alpha\} = \begin{cases}\emptyset&\text{if
  }F(\eta) > \alpha\\ \,\{\xi\}&\text{if
  }F(\eta) = \alpha \end{cases}
\end{equation*}
%$F(\xi) = \alpha$ and $\{x \in \R : F(x) = \alpha\} \not= \{\xi\}$;
\end{itemize}
In particular, $F(\xi) = \alpha$ if and only if $\{x \in \R : F(x) = \alpha\} \not= \emptyset$, and
% If $\{x \in \R : F(x) = \alpha\}$ just consists of one element $x_0$, it follows that $\eta = \xi = x_0$.
$\eta - \xi = \Delta F^{\wedge}(\alpha) = 0$ if and only if $\{x \in \R : F(x) = \alpha\} \in \{\emptyset, \{\xi\}\}$, and if $\eta > \xi$, then $\Delta F(\eta) = 0$ if and only if $F(\eta) = \alpha$. 
%Moreover,
%\begin{eqnarray*}
%\R \setminus F^\wedge((0, 1)) \cup \R \setminus F^\vee((0, 1)) \subseteq \{x \in \R : F \text{ is continuous in } x\}\,.
%J_F & = & \big\{x \in \R :  F^{\wedge}(u) = x = F^{\vee}(u) \text{ and } \Delta F(F^{\wedge}(u)) > 0 \text{ for some } u \in (0, 1)\big\}\\
%& = & \big\{x \in \R :  F^{\wedge}(u) = x = F^{\vee}(u) \text{ and } \Delta F(F^{\vee}(u)) > 0 \text{ for some } u \in (0, 1)\big\}\,.
%& = & J_F^{+} \cup J_F^{-}\,,
%\end{eqnarray*}
%where $J_F^{+} : = \{x \in \R : x = F^{\wedge}(u) \text{ and } F(x) > u \text{ for some } u \in (0, 1)\}$ and $J_F^{-} : = \{x \in \R : x = F^{\vee}(u) \text{ and } F(x-) < u \text{ for some } u \in (0, 1)\}$.
\end{corollary}
\begin{remark}\label{BCBS_and_newsvendor_problem}
Let $\alpha \in (0,1)$. Then, according to \cite[Corollary 1.1]{ACS2007} for a large class of distribution functions $F$ any non-empty set $[\xi, \eta] = [F^\wedge(\alpha), F^\vee(\alpha)]$ even emerges as a set of optimal solutions of the so called ``single period newsvendor problem'' which asks for the minimisation of coherent risk measures, such as the conditional-value-at-risk (which coincides with Expected Shortfall), corresponding to a cost function, induced by random demand. Here, one should recall that recently the Basel Committee on Banking Supervision (BCBS) suggested in their updated consultative document ``Fundamental review of the trading book'' to implement Expected Shortfall at $\alpha = 97.5\%$ in a bank's internal market risk model to calculate its minimum capital requirements with respect to market risk.
\end{remark}  
  
Let ${\mathcal{B}}(\R)$ denote the set of all Borel subsets of $\R$. In the following, let $\mu_F : {\mathcal{B}}(\R) \longrightarrow [0, \infty]$ be the Lebesgue-Stieltjes measure of $F$. For a detailed description of the construction and properties of the Lebesgue-Stieltjes measure (including Lebesgue-Stieltjes integration), we refer the reader to e.\,g. \cite{AD2000} and \cite{B1995}. For the convenience of the reader, we recall the following fundamental result (cf. \cite[Theorem 12.4]{B1995}):
\begin{theorem}[Lebesgue-Stieltjes measure]
Let $G: \R \longrightarrow \R$ be an arbitrary non-decreasing and right-continuous function. Then there exists a unique Borel measure $\mu_G$ satisfying
\[
\mu_G\big( (x, y]\big) = G(y) - G(x)
\]
for all $x, y \in \R$.
\end{theorem}
Clearly, this crucial result implies that $\mu_G\big( (x, y)\big) = G(y-) - G(x)$ and hence
\[
\mu_G\big( \{y\}\big) = \mu_G\big((x, y]\big) - \mu_G\big( (x, y)\big) = G(y) - G(y-) = \Delta G(y)
\]
for all $y \in \R$. Moreover, $\mu_G(\R) = 0$ if and only if $G$ is a constant function on $\R$.
%\\[0.5em]

Returning to our distribution function $F$, a direct application of $\mu_F$ leads to another important implication of Lemma \ref{thm:flat_pieces_of_F}:
\begin{corollary}\label{thm:flat_pieces_of_F_have_LS_measure_0}
Let $0 < \lambda \leq 1$ and $\alpha \in (0, 1)$. Then $A_{\lambda, \alpha}^{+} \in {\mathcal{B}}(\R)$, and
%is a Borel set on $\R$, and
\[
\mu_F\big(A_{\lambda, \alpha}^{+}\big) = 0\,.
\]
In particular, if $\xi < \eta$, then
\[
\mu_F\big(\{x \in \R : F(x) = \alpha\}\big) = \Delta F(\xi) = \alpha - F(\xi-)\,.
\]
\end{corollary}
\begin{proof}
Nothing is to prove if $A_{\lambda, \alpha}^{+} = \emptyset$. So, let $A_{\lambda, \alpha}^{+} \not= \emptyset$. Then $\eta - \xi = \Delta F^{\wedge}(\alpha) > 0$.
%The function $F$ can either jump at $\eta$ ( which is equivalent to $\eta \notin A_{\lambda, \alpha}^{+}$) or not. So, we have to consider these two possible cases.
%\\[0.5em]

Suppose first that $F(\eta) > \alpha$. Then
\[
A_{\lambda, \alpha}^{+} = (\xi, \eta) = \bigcup_{n = 1}^{\infty}\big(\xi, \eta - \frac{1}{n}\big]\,.
\]
Consequently, since in general $F(x) = \alpha = F(\xi)$ for all $x \in (\xi, \eta)$, it follows that
\[
\mu_F\big(A_{\lambda, \alpha}^{+}\big) = \lim_{n \to \infty}\mu_F\Big(\big(\xi, \eta - \frac{1}{n}\big]\Big) = \lim_{n \to \infty}\big(F(\eta - \frac{1}{n}) - F(\xi)\big) = \alpha - \alpha = 0\,.
\]
Now suppose that $F(\eta) = \alpha$. Then $\eta \in A_{\lambda, \alpha}^{+}$, and it follows that $F$ is continuous at $\eta$. Thus, $\mu_F(\{\eta\}) = \Delta F(\eta) = 0$. Since in this case
\[
A_{\lambda, \alpha}^{+} = (\xi, \eta) \dotcup \{\eta\}\,,
\]
it consequently follows that
\[
\mu_F\big(A_{\lambda, \alpha}^{+}\big) = \lim_{n \to \infty}\big(F(\eta - \frac{1}{n}) - F(\xi)\big) + \mu_F(\{\eta\}) = \alpha - \alpha + 0 = 0\,.
\]
\end{proof}
Next, we are going to reveal in detail that the function $F$ is almost ``left-invertible'' at every $x \in \R$ which does not belong to the preimage $A_{\lambda, \alpha}^{+}$ of a ``flat piece'' of $F$. More precisely:
%REMARK: In the following Theorem, we could avoid the assumption ``$\mu_F(\{x \in \R: F_\lambda(x) = 0\}) = 0$, $\mu_F(\{x \in \R: F_\lambda(x) = 1\}) = 0$'' if we work on $D$ already and hence make use of the $D$-induced Borel algebra ${\mathcal{B}}(D) = {\mathcal{B}}(\R) \cap D$! Note that the complement of a Borel set $A \cap D$ has to be taken in $D$, implying statements like ``... if $x \in D \setminus (A \cap D) = D \cap A^c$, then ...'', where of course $D^c : = \R \setminus D$!}
\begin{theorem}\label{thm:inversion_almost_surely}
Let $0 < \lambda \leq 1$. Assume that $0 < F_\lambda < 1$ $\mu_F$-almost everywhere. Then
\[
{\text{id}}\,_{\R} = F^{\wedge} \circ F_\lambda \hspace{1cm} \mu_F\text{-almost everywhere}\,.
\]
In particular, if $0 < F < 1$ $\mu_F$-almost everywhere, then
\[
{\text{id}}\,_{\R} = F^{\wedge} \circ F \hspace{1cm} \mu_F\text{-almost everywhere}\,.
\]
% and $0 < \mu_F(\R) < \infty$\footnote{For example, if there were some $\alpha_0 \in (0, 1)$ such that $\lim\limits_{x \to \infty} F(x) > \alpha_0$, then $\mu_F(\R) \geq \mu_F\big([F^{\wedge}(\alpha_0), \infty )\big) > 0$ - as it is the case if $F$ were a distribution function.},
% In particular, $x = F^{\wedge}\big(F_\lambda(x)\big)$ $\mu_F$-almost everywhere.
\end{theorem}
\begin{proof}
Let $0 < \lambda \leq 1$. Consider the Borel set
%\begin{equation}\label{eqn:the_null_set}
% N : = \begin{cases}\emptyset&\text{if
%  }J_{F^{\wedge}} = \emptyset\\ \bigcup_{\alpha \in  J_{F^{\wedge}}} A_{\lambda, \alpha}^{+} &\text{if }J_{F^{\wedge}} \not= emptyset \,,\end{cases}
%\end{equation}
\[
N_\lambda : = \{x \in \R : F_\lambda(x) = 0\} \dotcup \{x \in \R : F_\lambda(x) = 1\} \dotcup \bigcup_{\alpha \in  J_{F^{\wedge}}} \hspace{-0.2cm} A_{\lambda, \alpha}^{+}\,,
\]
where $J_{F^{\wedge}} : = \{\alpha \in (0, 1) : \Delta F^{\wedge}(\alpha) > 0 \}$ denotes the set of all jumps of the function $F^{\wedge}$.\footnote{Note that by construction $N_\lambda = \{x \in \R : F_\lambda(x) = 0\} \dotcup \{x \in \R : F_\lambda(x) = 1\}$ if $J_{F^{\wedge}} = \emptyset$.}
% $N_\lambda : = \{x \in \R : F_\lambda(x) = 0\} \dotcup \{x \in \R : F_\lambda(x) = 1\}$.
Since the (left-continuous) function $F^{\wedge} : (0, 1) \longrightarrow \R$ is non-decreasing, $J_{F^{\wedge}}$ is at most countable. Hence, if $J_{F^{\wedge}} \not= \emptyset$, there exists a subset $\M$ of $\N$, and a sequence $(\alpha_n)_{n \in \M}$, consisting of pairwise distinct elements $\alpha_n \in J_{F^{\wedge}}$, such that $J_{F^{\wedge}} = \{\alpha_n : n \in \N\} $. Thus, $\bigcup_{\alpha \in  J_{F^{\wedge}}}{}A_{\lambda, \alpha}^{+} = \ldju{n \, \in \, \M}{}\,A_{\lambda, \alpha_n}^{+}$. Corollary \ref{thm:flat_pieces_of_F_have_LS_measure_0} therefore implies that - in any case - $\mu_F(N_\lambda) = 0$ and hence $\R \setminus N_\lambda \not= \emptyset$ (since $F$ cannot be a constant function on the whole real line).
%\\[0.5em]

Let $x \in \R \setminus N_\lambda$. Put $\alpha(x) : = F_\lambda(x)$. Then $0 < \alpha(x) < 1$, and $x \in \bigcap_{\alpha \in  J_{F^{\wedge}}} \big(\R \setminus A_{\lambda, \alpha}^{+}\big)$. Thus, $\xi(x) : = F^{\wedge}(\alpha(x))$ is well-defined. Consider $\eta(x) : = F^{\vee}(\alpha(x)) = F^{\wedge}(\alpha(x)+)$.
%\\[0.5em]

First, let $J_{F^{\wedge}} = \emptyset$. Then $\xi(x) = \eta(x)$. Lemma \ref{thm:flat_pieces_of_F} therefore implies that $A_{\lambda, \alpha(x)}^{+} = \emptyset$. In particular, $x \notin A_{\lambda, \alpha(x)}^{+}$. Hence, since $F_\lambda(x) = \alpha(x) \ngtr \alpha(x)$, it consequently follows that
\[
x \leq \xi(x) = F^{\wedge}(\alpha(x)) = F^{\wedge}(F_\lambda(x))\,,
\]
and hence $x = F^{\wedge}(F_\lambda(x))$.
%\\[0.5em]

Now let $J_{F^{\wedge}} \not= \emptyset$. If $\alpha(x) \notin J_{F^{\wedge}}$, it follows again that $\xi(x) = \eta(x)$ and hence
\[
x \leq \xi(x) = F^{\wedge}(\alpha(x)) = F^{\wedge}(F_\lambda(x)) \leq x\,,
\]
as above. So, let $\alpha(x) \in J_{F^{\wedge}}$.
%; i.\,e., $\xi(x) < \eta(x)$. Then
%\[
%\alpha(x) \in J_{F^{\wedge}} : = \{\alpha \in (0, 1) : \Delta F^{\wedge}(\alpha) > 0 \}\,.
%\]
%Since $F^{\wedge}$ is non-decreasing as well, $J_{F^{\wedge}}$ at most is a countable subset of $(0, 1)$. Hence, there exists a subset $\M$ of $\N$, and a sequence $(\alpha_n)_{n \in \M}$, consisting of elements $\alpha_n \in J_{F^{\wedge}}$, such that $J_{F^{\wedge}} = \bigcup_{n \in \M} \{\alpha_n\}$.
Then $\alpha(x) = \alpha_m$ for some $m \in \M$,
%Without loss of generality, we may assume that $\alpha_n \not= \alpha_k$ for all $k \not= n$.
%Now consider the non-empty set
%\[
%A : = \cdju{n \in \M}{} A_{\lambda, \alpha_n}^{+}\,.
%\]
and hence $A_{\lambda, \alpha(x)}^{+} = A_{\lambda, \alpha_m}^{+}$.
%Due to Lemma \ref{thm:flat_pieces_of_F_have_LS_measure_0}, it follows that $\mu_F(A) = 0$, implying in particular that $\R \setminus A \not= \emptyset$ (since $\mu_F(\R) = 1 - 0 > 0$).
Since $x \in \R \setminus N_\lambda \subseteq \R \setminus A_{\lambda, \alpha(x)}^{+}$, it follows once more again that $x \leq \xi(x) = F^{\wedge}(\alpha(x))$, and hence
% or $F_\lambda(x) = \alpha(x) > \alpha(x)$. Thus,
\[
x = \xi(x) = F^{\wedge}(\alpha(x)) = F^{\wedge}(F_\lambda(x))\,.
\]
\end{proof}
Next, we consider the set $A_{\lambda, \alpha}^{\sim}$. Again, in line with \cite{R2009}, we put $q : = F(\xi-)$ and $\beta : = \Delta F(\xi) \geq 0$. Then
\[
q + \beta = F(\xi) \stackrel{\eqref{eqn:quantile_ineq}}{\geq} \alpha \geq q\,.
\]
Obviously, we may write:
\begin{remark}
$A_{\lambda, \alpha}^{\sim} = \big\{x \in \R : x = \xi \mbox{ and } \beta \lambda \leq \alpha - q\big\}$.
% = \big\{x \in \R : (x, \lambda) \in \{\xi\} \times \{v \in (0, 1] : v\beta \leq
% \alpha - q \}\big\}
% \{\xi\} \cap \{x \in \R : \beta \lambda \leq \alpha - q  \}$.
% $A_{\lambda, \alpha}^{\sim} = \{\xi\} \cap \{x \in \R : \beta \lambda \leq \alpha - q  \}$.
%\\[1em]
% In particular, either $A_{\lambda, \alpha}^{\sim} = \{\xi\}$, or $A_{\lambda, \alpha}^{\sim} = \emptyset$.
\end{remark}
%Hence,
%\begin{equation}\label{eqn:input_of_V}
%x \in A_{\lambda, \alpha}^{\sim} \,\, \mbox{ if and only if } \, (x, \lambda) \in \{\xi\} \times \{l \in (0, 1] : l\beta \leq \alpha - q \}\,.
%\end{equation}
%\begin{proof}
%Again, that's what you should do now.
%\qed
%\end{proof}
Moreover, by using a similar argument like that one which has shown us that the set $A_{\lambda, \alpha}$ is non-empty, we further obtain
\begin{remark}
$A_{\lambda, \alpha}^{-} = (-\infty, \xi) = \{x \in \R : x < \xi\}$.
\end{remark}
%\begin{proof}
%This follows by the argument which has shown us that the set $A_{\lambda, \alpha}$ %is non-empty.
% Let $x < \xi$. Then $x \leq \xi - \delta$ for some $\delta > 0$. Hence,
% $F_\lambda(x) \leq F(x) \stackrel{\eqref{eqn:quantile_strict_ineq}}{<} \alpha$.
%\qed
%\end{proof}
Observe that only the subset $A_{\lambda, \alpha}^{\sim}$ of $A_{\lambda, \alpha}$ does depend on the choice of $\lambda \in [0,1]$.
\subsection{The inclusion of randomness}
\label{sec:3}
In addition to our assumptions above, we now fix a given probability
space $(\Omega, \mathcal{F}, \P)$. Let $X : \Omega \longrightarrow \R$ and $V : \Omega \longrightarrow \R$ be two given random variables (on this probability space) such that $V \sim U(0,1)$ is uniformly distributed over $(0, 1)$ and independent of $X$. In the following we consider the random variable $F_V(X)$, defined on $\{V \in (0,1]\}$ as
\[
F_V(X)(\omega) : = F_{V(\omega)}(X(\omega)) = F_\lambda(x)\,,
\]
where here $\omega \in \Omega$, $\lambda : = V(\omega)$ and $x : = X(\omega)$\footnote{Since $V \sim U(0,1)$, we obviously have $\P(V \in (0,1])= 1$ and hence $\P(V \notin (0,1])= 0$.}.
% The next non-trivial fact (which in its generality is of its own interest) is a crucial element in the proof of Sklar's Theorem.
Next, we have to evaluate $\P \lb F_V(X) \leq \alpha \rb = \P \lb F_V(X) \leq \alpha \mbox{ and } V \in (0,1]\rb$; i.e, we wish to calculate
\[
\P \lb F_V(X) \leq \alpha \rb
% = \P \lb \left\{\omega \in \Omega: F_V(X)(\omega) \leq \alpha \right\}\rb
= \P \lb \left\{\omega \in \Omega: X(\omega) \in  A_{V(\omega), \alpha} \right\} \mbox{ and } V \in (0,1]\rb
\]
% A crucial example in the theory and applications of copulas is the random variable
% $F_X(X)$, where $F_X$ is the distribution function of $X$. We will meet this
% friend at the end of this paper.
Due to our previous observations, we have
\[
A_{V(\omega), \alpha} =  A_{V(\omega), \alpha}^{+} \dotcup A_{V(\omega), \alpha}^{\sim} \dotcup A_{V(\omega), \alpha}^{-}
\]
for all $\omega \in \{V \in (0,1]\}$. Consequently, given the assumed independence of $V$ and $X$, Lemma \ref{thm:flat_pieces_of_F} implies that\footnote{Here, $\{X \in A_{V, \alpha}\} : = \{\omega \in \Omega: X(\omega) \in A_{V(\omega), \alpha}\}$ and $\{X \in A_{V, \alpha}^{i}\} : = \{\omega \in \Omega: X(\omega) \in A_{V(\omega), \alpha}^{i}\}$, where $i \in \{+, \sim, -\}$.
}:
\begin{eqnarray*}\label{eqn:Prob_value}
\P \lb F_V(X) \leq \alpha \rb & \!\!\! = \!\!\! & \P(X \in A_{V, \alpha}^{+}\mbox{ and } V \in (0,1]) + \P(X \in A_{V, \alpha}^{\sim}\mbox{ and } V \in (0,1])\\
& + & \P(X \in A_{V, \alpha}^{-}\mbox{ and } V \in (0,1])\\
%& = & \P \lb X > \xi \mbox{ and } F(X) = \alpha \rb + \P(X = \xi \mbox{ and } \beta V \leq \alpha - q\mbox{ and } V \in (0,1]) + \P(X \in A_{V, \alpha}^{-}\mbox{ and } V \in (0,1])\\
& \!\!\! = \!\!\! & \P \lb X > \xi \mbox{ and } F(X) = \alpha \rb + \P \lb X = \xi \rb \P \lb \beta V \leq \alpha - q \mbox{ and } V \in (0,1]\rb + \P \lb X < \xi \rb\,.
\end{eqnarray*}
%\begin{remark}
%Observe, that we have not yet assumed that $F$ is a distribution function.
%\end{remark}
\noindent Apparently, to continue with the calculation of the respective probabilities, we have to consider the following two possible cases: $\beta = 0$ and $\beta > 0$:
\begin{itemize}
\item[(i)] Let $\beta = 0$. Thus, since $\alpha - q \geq 0$, it follows that
% Then $F$ is continuous at $\xi$. Hence, $F(\xi) = \alpha$, and it follows that
\[
\P \lb F_V(X) \leq \alpha \rb = \P \lb X > \xi \mbox{ and } F(X) = \alpha \rb + \P \lb X \leq \xi \rb
\]
\item[(ii)] Let $\beta > 0$. Since $V \sim U(0,1)$ is uniformly distributed over $(0, 1)$, we have\newline $\P \lb \beta V \leq \alpha - q \mbox{ and } V \in (0,1]\rb = \P \lb V \leq \frac{\alpha - q}{\beta } \rb = \frac{\alpha - q}{\beta}$. Hence, since $\frac{\alpha - q}{\beta} - 1 = \frac{\alpha - F(\xi)}{\beta}$, it follows that
\[
\P \lb F_V(X) \leq \alpha \rb = \P \lb X > \xi \mbox{ and } F(X) =  \alpha \rb + \lb \frac{\alpha - F(\xi)}{\beta} \rb \P \lb X = \xi \rb + \P \lb X \leq \xi \rb\,.
\]
\end{itemize}
Moreover, by taking into account that $F(\xi) = \alpha$ in case (i) (since $F$ is continuous at $\xi$ if $\beta =0$), we have arrived at the following important
\begin{lemma}\label{thm:prob_calc_and_rv_as_gen_inverse_value}
%\label{thm:prob_calc}
Suppose that $F : \R \longrightarrow [0,1]$ is an arbitrary distribution function. Let $\alpha \in (0, 1)$. Put $\xi : = F^{\wedge}(\alpha)$ and
%, and assume that $\xi < \infty$. Put
$\beta : = \Delta F(\xi)$. Let $X, V$ be two random variables, both defined on the same probability space $(\Omega, {\mathcal{F}}, \P)$, such that $V \sim U(0,1)$ and $V$ is independent of $X$. Then
\[
\P \big(F_V(X) \leq \alpha \big) - \alpha = \P \big( X > \xi \mbox{ and } F(X) = \alpha \big) + c_{\beta} \big( \P \big( X = \xi \big) - \beta \big) + \big( \P \big( X \leq \xi \big) - F(\xi) \big)\,,
\]
where $c_{\beta} : = 0$ if $\beta = 0$ and $c_{\beta} : = \frac{\alpha - F(\xi)}{\beta}$ if $\beta \not= 0$.
\end{lemma}
To conclude, let us slightly point towards the fact that Lemma \ref{thm:prob_calc_and_rv_as_gen_inverse_value} could also be viewed as a building block of a probabilistic limit theorem (whose detailed discussion would then exceed the main goal of this paper, though).
% Actually, Lemma \ref{thm:prob_calc} does not only immediately imply Proposition
% 2.1 in \cite{R2009}. As we will see later, it also gives us a limit theorem
% (``uniformly in probability'').
\subsection{The role of the distribution function of $X$}
\label{sec:4}
From now on, $F : =  F_X = \P(X \leq \cdot)$ is given as the distribution function of a given random variable $X$. 
%Since then $\lim_{x \to -\infty} F(x) = 0$ and $\lim_{x \to \infty} F(x) = 1$, it follows that in fact for any $0 < \alpha < 1$ the set $\{x \in \R : F(x) \geq \alpha\}$ is non-empty and bounded from below, implying that $\vert F^{\wedge}(\alpha) \vert < \infty$ for all $\alpha \in (0, 1)$.
\begin{proposition}\label{thm:uniform_distribution}
Let $X, V$ be two random variables, both defined on the same probability space $(\Omega, {\mathcal{F}}, \P)$, such that $V \sim U(0,1)$ and $V$ is independent of $X$. Let $F = F_X$ be the distribution function of $X$. Then $F_V(X) \sim U(0,1)$ is a uniformly distributed random variable. Moreover,
\[
\P\big(F(X) \leq \alpha\big) = \alpha = \P\big(X \leq F^{\wedge}(\alpha)\big)
\]
on the set $\big\{ \alpha \in (0,1) : F^{\wedge}(\alpha) < F^{\vee}(\alpha) \big\}$.
\end{proposition}
\begin{proof}
Let $0 < \alpha < 1$. Lemma \ref{thm:prob_calc_and_rv_as_gen_inverse_value} - applied to $F = F_X$ - directly leads to
\[
\P \big(F_V(X) \leq \alpha \big) - \alpha = \P \big( X > \xi \mbox{ and } F(X) = \alpha  \big)\,.
\]
Corollary \ref{thm:flat_pieces_of_F_have_LS_measure_0} further implies that for any $0 < \lambda \leq 1$ we have
\[
\P\big( X > \xi \mbox{ and } F(X) = \alpha  \big) = \P\big(X \in  A_{\lambda, \alpha}^{+}\big) = \mu_F\big(A_{\lambda, \alpha}^{+}\big) = 0\,.
\]
Thus, we have
\begin{equation}\label{eqn:uniform_distr}
\P\big(F_V(X) \leq \alpha \big) = \alpha \text{ for any } 0 < \alpha < 1\,.
\end{equation}
Consequently, $\sigma$-additivity of the probability measure $\P$ allows one to continuously extend \eqref{eqn:uniform_distr} to the whole real line. Hence, $F_V(X) \sim U(0,1)$ is uniformly distributed.
%\\[0.5em]

Now let $\alpha \in (0,1)$ such that $\xi : = F^{\wedge}(\alpha) < F^{\vee}(\alpha) = : \eta$. Since $F$ is the distribution function of $X$, we have $\mu_F = \P(X \in \cdot)$. Thus, Corollary \ref{thm:flat_pieces_of_F_have_LS_measure_0} leads to
\[
\P\big(F(X) = \alpha \big) = \mu_F\big(F = \alpha\big) = \alpha - F(\xi-) = \P(X = \xi)\,.
\]
Since always
\[
\P\big(F(X) < \alpha \big) = \P(X < \xi)\,,
\]
it follows that
\[
\P\big(F(X) \leq \alpha \big) = \P(X \leq \xi) = F(\xi) = \alpha\,,
\]
and we are done.
\end{proof}
\begin{remark}
It is well-known that in the case of a \textit{continuous} distribution function, $G = G_X$ say, $G(X)$ is uniformly distributed over $(0,1)$ (cf. e.g. \cite[Proposition 3.1]{EH2013}). However, continuity of $G$ is even a necessary condition for $G(X)$ being uniformly distributed over $(0,1)$. Else there were some $x_0 \in \R$ such that 
\[
0 < \Delta G(x_0) = G(x_0) - G(x_0-) = \P(X = x_0) \leq \P(G(X) = G(x_0)) = 0
\]
which would be a contradiction. Consequently, if $F$ has non-zero jumps, $F_V(X)$ still would be uniformly distributed over $(0,1)$ as opposed to $F(X)$.      
\end{remark}
In order to complete the proof of statement of Proposition 2.1 in \cite{R2009}, let us recall that the assumed independence of the random variables $X$ and $V$ implies that the bivariate distribution function $F_{(V, X)}$ of the random vector $(V, X)$ coincides with the product of the distribution functions $F_V$ and $F_X$. Moreover, since $V \sim U(0,1)$, it follows that on ${\mathcal{B}}\big((0,1]\big)$ $\mu_{F_V} = \P(V \in \cdot)$ coincides with the Lebesgue measure $m$. Hence, if $\Phi : {\R}^2 \longrightarrow \R$ denotes an arbitrary non-negative (or bounded) Borel function on $\big({\R}^2 , {\mathcal{B}}({\R}^2), \mu_{F_V} \otimes \mu_{F_X} \big)$, an immediate application of the Fubini-Tonelli Theorem leads to
\begin{equation}\label{eqn:how_to_ignore_V}
\E\big[\Phi(V, X)\big] = \E\big[\Phi(V, X)\ind_{\{V \in (0, 1]\}}\big]
% = \E\big[ \E\big[\Phi(V, X) \vert V \big] \big]
%= \E\big[\varphi(V)\big]
= \int_{0}^{1}\E\big[\Phi(\lambda, X)\big]\, m(\text{d}\lambda) = \int_{0}^{1}\Big(\int_{\R}\Phi(\lambda, x)\, \mu_F(\text{dx}) \Big)m(\text{d}\lambda)
\end{equation}
\begin{theorem}\label{thm:Rueschendorf}
Let $X, V$ be two random variables, defined on the same probability space $(\Omega, {\mathcal{F}}, \P)$, such that $V \sim U(0,1)$ and $V$ is independent of $X$. Let $F =F_X$ be the distribution function of the random variable $X$.
% Put $U : = F_V(X)$.
Then 
\[
X = F^{\wedge}\big(F_V(X)\big) = F^{\wedge}\big(F(X-) + V \Delta F(X)\big) \hspace{0.5cm} \P\text{-almost surely}.
\]
If in addition $\P\big(0 < F(X) < 1\big) = 1$ (for example, if $F$ is continuous), then
\[
X = F^{\wedge}\big(F(X)\big) \hspace{0.5cm} \P\text{-almost surely}\,.
\]
\end{theorem}
\begin{proof}
%Since $F$ is the distribution function of $X$, $F \not= \text{const}$ on $\R$.
%it follows that $\mu_F = \P(X \in \cdot)$ and hence $\mu_F(\R) = 1 > 0$. 
Let $B_\lambda : = \{x \in \R : F_\lambda(x) = 0\}$, where $0 < \lambda \leq 1$. Then
\[
\P\big( F_\lambda(X) = 0 \big) = \E\big[\ind_{B_\lambda}(X)\big]\,.
\]
On the other hand, equality \ref{eqn:how_to_ignore_V} clearly implies
\[
\P\big( F_V(X) = 0 \big) = \int_{0}^{1}\E\big[\ind_{B_\lambda}(X)\big]\,m(\text{d}\lambda)\,.
\]
Hence, since $F_V(X) \sim U(0,1)$,
% (due to Proposition \ref{thm:uniform_distribution}),
it follows that $\int_{0}^{1}\E\big[\ind_{B_\lambda}(X)\big]\,m(\text{d}\lambda) = 0$, implying that for $m$-almost all $\lambda \in (0, 1]$ we have
\[
\mu_F\big(F_\lambda = 0 \big) = \P\big( F_\lambda(X) = 0 \big) = \E\big[\ind_{B_\lambda}(X)\big] = 0\,.
\]
Similarly, we obtain
\[
\mu_F\big(F_\lambda = 1 \big) = 0
\]
for $m$-almost all $\lambda \in (0, 1]$. Hence, there exists an $m$-null set $L \in {\mathcal{B}}\big((0, 1]\big)$ such that 
%$\mu_F = \P(X \in \cdot)$ satisfies all requirements of Theorem \ref{thm:inversion_almost_surely} 
$0 < F_\lambda < 1$ $\mu_F$-almost everywhere for all $\lambda \in (0,1]\setminus L =: A$.
%\\[0.5em]

Thus, given the construction in the proof of Theorem \ref{thm:inversion_almost_surely}, it follows that for all $\lambda \in A$ there exists a $\mu_F$-Borel null set $N_\lambda$, such that
% for all $\lambda \in \R \setminus L$ and
for any $x \in \R\setminus N_\lambda$ the value $F^{\wedge}\big(F_\lambda(x)\big)$ is well-defined and satisfies $F^{\wedge}\big(F_\lambda(x)\big) = x$.
%Thus (by canonical extension to $0$), there is a non-negative Borel function $\Phi$ on $\big((0, 1] \times \R, {\mathcal{B}}((0,1]) \times {\mathcal{B}}(\R) \big)$, satisfying
%\[
%\Phi(\lambda, x) = 0 = x - F^{\wedge}\big(F_\lambda(x)\big)
%\]
%for all $\lambda \in \R\setminus L$ and $x \in \R\setminus N_\lambda$.
%Hence, if we apply equality \eqref{eqn:how_to_ignore_V} to the function $\Phi \equiv 0$, it follows that
%\[
%\E\big[\Phi(V, X)] = 0\,,
%\]
%implying that there is a set $D \in {\mathcal{F}}$, satisfying $\P(D) = 0$ and $\Phi(V, X) \equiv 0$ on $\Omega \setminus D$.
Hence, since
\[
\P\big(X \in N_V \mbox{ and } V \in A \big) \stackrel{\eqref{eqn:how_to_ignore_V}}{=} 
%\int_{0}^{1} \P\big(X \in N_\lambda \big) m(\text{d}\lambda)
\int_{A} \P\big(X \in N_\lambda \big) m(\text{d}\lambda) = \int_{A} \mu_F\big(N_\lambda\big) m(\text{d}\lambda) = 0\,,
\]
it consequently follows $X = F^{\wedge}\big(F_V(X)\big) = F^{\wedge}\big(F(X-) + V \Delta F(X)\big)$ on the set $\Omega \setminus N \subseteq \{0 < V \leq 1\}$, where $N : = \{V \notin A\} \dotcup \{X \in N_V \mbox{ and } V \in A\}$.
%\\[1em]

If in addition $\P(0 < F(X) < 1) = 1$, $F^{\wedge}(F(X))$ is well-defined $\P$-a.\,s. Consequently, since also $F^\wedge$ is non-decreasing, there exists a $\P$-null set $\tilde{N}$, satisfying $\Omega\setminus\tilde{N} \subseteq \Omega\setminus N \subseteq \{0 < V \leq 1\}$, such that 
\[
X(\omega) = F^{\wedge}\big(F(X(\omega)-) + V(\omega) \Delta F(X(\omega))\big) \leq F^{\wedge}(F(X(\omega))) \leq X(\omega)
\]
for all $\omega \in \Omega\setminus\tilde{N}$.
\end{proof}
\begin{remark}
One might be easily lead to assume that already a direct application of Proposition \ref{thm:in_the_quantiles_range} implies the first statement of Theorem \ref{thm:Rueschendorf}. However, in the first instance Proposition \ref{thm:in_the_quantiles_range} only implies that the equality $X = F^{\wedge}\big(F_V(X)\big) = F^{\wedge}\big(F(X-) + V \Delta F(X)\big)$ at least holds on the set $D : = \{\omega \in \Omega : \Delta F(X(\omega)) > 0 \mbox{ and } 0 < V(\omega) < 1\}$. Now consider $N : = \Omega \setminus D$. Then 
\[
\P(N) = \P(\{\Delta F(X) = 0\}) = \P(\{\Delta F_V(X) = 0\}) = \P(U = Y)\,,
\]
where $U : = F_V(X) \sim U(0,1)$ and $Y : = F_V(X-) = U - V \Delta F(X)$. However, in general we don't know whether $U$ is independent of $Y$.   
\end{remark}
For the convenience of the reader, we conclude our paper with a full proof of the \textit{general} version of Sklar's Theorem, built on Theorem \ref{thm:Rueschendorf} (cf. also the proof of Theorem 1.2 in \cite{MS2012}, respectively the short proof of Lemma 3.2 in \cite{MS1975}), complemented with another interesting and seemingly novel observation (Remark \ref{thm:obsv_on_Sklars_thm}), induced by Lemma \ref{thm:flat_pieces_of_F}.
\begin{corollary}[Sklar's Theorem]\label{thm:Sklar}
Let $n \in \N$ and $F_{(X_1, \ldots, X_n)}$ be a joint $n$-variate distribution function of a random vector $(X_1, X_2, \ldots, X_n) : \Omega \longrightarrow \R^n$ with marginals $F_i : = F_{X_i}$ $(i =1,2, \ldots, n)$. Then there exist a copula $C_F$ such that for all $(x_1, x_2, \ldots, x_n) \in {{\R}^n}$
\[
F_{(X_1, \ldots, X_n)}(x_1, x_2, \ldots, x_n) = C_F(F_1(x_1), F_2(x_2), \ldots, F_n(x_n))\,.
\]
If all $F_i$ are continuous, then the copula $C_F$ is unique. Otherwise, $C$ is uniquely determined on $\prod_{i=1}^{n} F_i(\R)$. Conversely, if $C$ is a copula and $H_1, H_2, \ldots, H_n$ are distribution functions, then the function $F$ defined by
\[
F(x_1, x_2, \ldots, x_n) : = C(H_1(x_1), H_2(x_2), \ldots, H_n(x_n))
\]
is a joint distribution function with marginals $H_1, H_2, \ldots, H_n$.
\end{corollary}
\begin{proof}
Let $i \in \{1,2, \ldots, n\}$ and $V_i \sim U(0,1)$. On $\{0 < V_i \leq 1\}$ put $U_i : = V_i \Delta F_i(X_i) + F_i({X_i}-)$. According to Theorem \ref{thm:Rueschendorf} there exist null sets $M_1, M_2, \ldots, M_n \in \mathcal{F}$, such that on $\Omega\setminus M_i \subseteq \{0 < V \leq 1\}$ $Z_i : = {F_i}^{\wedge}(U_i)$ is well-defined and satisfies $X_i \equiv Z_i$ for every $i \in \{1,2, \ldots, n\}$. Thus, $\P(M) = 0$, where $M : = \bigcup_{i=1}^n M_i$.
%\\[0.5em] 

Let $F_{(X_1, \ldots, X_n)}(x_1, x_2, \ldots, x_n) : = \P(X_1 \leq x_1, X_2 \leq x_2, \ldots, X_n \leq x_n)$ denote the $n$-variate distribution function of the random vector $(X_1, X_2, \ldots, X_n)$. Consider the copula
\[
C_F(\gamma_1, \gamma_2, \ldots, \gamma_n) : = \P( U_1 \leq \gamma_1, U_1 \leq \gamma_2, \ldots, U_n \leq \gamma_n)\,,
\]
where $(\gamma_1, \gamma_2,\ldots, \gamma_n) \in [0,1]^n$. Since
%$Z_i \leq X_i$ (due to Remark \ref{thm:inversion_inequality_on_Omega}) and
\[
\{u \in (0,1 ): {F_i}^{\wedge}(u) \leq x_i\} = \{u \in (0,1): u \leq F_i(x_i)\}
\]
for all $i \in \{1,2, \ldots, n\}$ and $\P(M) = 0$, it consequently follows
% firstly follows that
%\begin{equation}\label{eqn:Sklar_inequality_on_Omega}
%F(x_1, x_2, \ldots, x_n) \leq \P(Z_1 \leq x_1, Z_2 \leq x_2, \ldots, Z_n \leq x_n) = c(F_1(x_1), F_2(x_2), \ldots, F_n(x_n))\,.
%\end{equation}
%According to Corollary \ref{thm:Rueschendorf} there exist Borel null sets $N_1, N_2, \ldots, N_n \subseteq \Omega$ satisfying $X_i = Z_i $ on $\R\setminus N_i$ for all $i$. Thus, $\P(N) = 0$, where $N : = \bigcup_{i=1}^n N_i$, and
\begin{align*}
{} & C_F(F_1(x_1), F_2(x_2), \ldots, F_n(x_n)) = \P \big( \{ Z_1 \leq x_1, Z_2 \leq x_2, \ldots, Z_n \leq x_n \} \cap \R \setminus M \big)\\
= \,\, & \P \big( \{ X_1 \leq x_1, X_2 \leq x_2, \ldots, X_n \leq x_n \} \cap \R \setminus M \big) = F_{(X_1, \ldots, X_n)}(x_1, x_2, \ldots, x_n)\,.
\end{align*}
\end{proof}
Combining Sklar's Theorem with Lemma \ref{thm:flat_pieces_of_F}, we immediately obtain another interesting result:
\begin{remark}\label{thm:obsv_on_Sklars_thm}
Let $(\alpha_1, \alpha_2, \ldots, \alpha_n) \in (0,1)^n$, satisfying ${F_i}^{\wedge}(\alpha_i) < {F_i}^{\vee}(\alpha_i)$ for all $i \in \{1,2, \ldots, n\}$. Then
\[
C_F(\alpha_1, \alpha_2, \ldots, \alpha_n) = F_{(X_1, X_2, \ldots, X_n)}\big({F_1}^{\wedge}(\alpha_1), {F_2}^{\wedge}(\alpha_2), \ldots, {F_n}^{\wedge}(\alpha_n)\big)\,.
\]
\end{remark}

\noindent\textbf{Acknowledgements} \vspace{0.2em}

\noindent The author would especially like to thank one of two anonymous referees for their particularly thorough perusal of the paper including some important remarks and comments and an indication of a few additional and very useful references. 
%\cite{KMP1999, MS2012} and \cite{SchS1983}.

\end{document}